\documentclass[11pt]{article}     
\usepackage{graphicx}
\usepackage{url}

\usepackage{color}
\definecolor{Red}{rgb}{1,0,0}

\definecolor{maroon}{rgb}{.69,.188,.376}
\def\qed{\hfill{$\blacksquare$}}

\definecolor{darkgreen}{rgb}{0,.5,0}
\definecolor{darkblue}{rgb}{0,0,.5}
\definecolor{magenta}{rgb}{1,0,1}

\usepackage[mathscr]{euscript}		

\usepackage{dsfont}

\usepackage[colorlinks=true]{hyperref}
\usepackage{amsthm}
\usepackage{cleveref}

\hypersetup{pdftex, colorlinks=true, linkcolor=maroon, citecolor=maroon,
  filecolor=blue,urlcolor=blue}

\usepackage{amsmath}

\usepackage{amssymb}
\usepackage{caption} 
\usepackage{anysize}
\usepackage{enumerate}
\usepackage{enumitem}

\usepackage[top=.95in, bottom = 1 in, left=1in, right = .6in]{geometry}

\newcommand{\remove}[1]{}

\def\F{\mathcal{F}}
\def\G{\mathcal{G}}
\def\X{\mathcal{X}}
\def\Y{\mathcal{Y}}
\def\H{\mathcal{H}}
\def\M{\mathcal{M}}

\def\P{\mathbb {P}}
\def\E{\mathbb {E}}
\def\bR{\mathbb{R}}
\newcommand{\EXP}[1]{\mathbb {E}\!\left(#1\right) }
\newcommand{\V}[1]{\mathsf{Var}\left( #1 \right)}
\newcommand{\C}[1]{\mathsf{Cov}\left( #1 \right)}

\newtheorem{theorem}{Theorem}
\newtheorem{corollary}{Corollary}
\newtheorem{lemma}{Lemma}

\newtheorem{remark}{Remark}

\numberwithin{equation}{section}
\numberwithin{theorem}{section}
\numberwithin{lemma}{section}
\numberwithin{corollary}{section}
\numberwithin{definition}{section}
\numberwithin{remark}{section}

\newcommand{\be}{\begin{equation}}
\newcommand{\ee}{\end{equation}}
\newcommand{\no}{\nonumber}

\def\cip{\underset{n \rightarrow \infty}{\overset{p}{\longrightarrow}}}

\def\cid{\underset{n \rightarrow \infty}{\overset{d}{\longrightarrow}}}
\def\1{\mathbf{1}}

\begin{document}
	
	\title{Central limit theorem for statistics of subcritical configuration models.}
	\author{Siva Athreya  \thanks{Research  supported in part by an ISF-UGC Project and a CPDA grant from the Indian Statistical Institute. 
  }  \and 
 D. Yogeshwaran \thanks{Research supported in part by INSPIRE Faculty Award and a CPDA grant from the Indian Statistical Institute.}}
 
               \maketitle
        
	       \begin{abstract}
 We consider subcritical configuration models and show that the
 central limit theorem for any additive statistic holds when the
 statistics satisfies a fourth moment assumption, a variance lower
 bound and the degree sequence of the graph satisfies a growth condition.
If the degree sequence
 is bounded, for well known statistics like component counts,
 log-partition function, and maximum cut-size which are Lipschitz
 under addition of an edge or switchings then the assumptions reduce
 to a linear growth condition for the variance of the statistic.
 Our proof is based on an application of the central limit theorem for
 martingale-difference arrays due to McLeish \cite{McLeish1974} to a
 suitable exploration process. \\
	\end{abstract}

{\em AMS Subject Class [2010] : Primary : 60F05, 05C80 ; Secondary :  60C05} 


\section{Introduction}

In this short note. we prove a central limit theorem for additive
statistics of random graphs that come from subcritical configuration
models. In a configuration model, we are given a {\em degree sequence}
$\{d^n_i\}_{i=1}^n, n \geq 1$ such that $m_n = \sum_{i=1}^nd^n_i$ is
even. One attaches $d^n_i$ half-edges or stubs to each vertex $i$, $1
\leq i \leq n.$ The random multi-graph formed by pairing uniformly at
random these half-edges or stubs is what is known as {\em the
  configuration model} $G_n := G(n,\{d^n_i\}_{i=1}^n$), \cite[Chapter 7]{Vander2016}. Under the assumption of subcriticality
(i.e. no giant component) and a growth condition on the degree
sequence we prove a central limit theorem for any additive statistic
of the graph having {an appropriate} variance lower bound. See
Section \ref{sec:modeldef} for a precise definition of the model and
assumptions, along with the statement of the main result (see Theorem
\ref{thm:main}). Our results can possibly be extended to a
larger class of random graph models which have similar constructions to
the configuration model (see Remark \ref{rthm:main}).

Graphs on $n$ vertices can be broadly divided into three classes:
Dense graphs, those with number of edges being of order~$n^2$; Sparse
graphs with bounded (average) degree and consequently having order $n$
edges; and in between are those whose average degree grows in the
number of vertices, but only at sub-linear speed. Each class has a
separate limiting theory. In this article we shall focus on a
particular model that falls in the sparse graph regime. One feature in this class of random graphs is the following phase transition. If the
expected degree (to be precise expectation of the size-biased degree
distribution) of a vertex is larger than one, namely the
super-critical phase, then there is a largest connected component
referred to as the ``giant'' component which contains a positive
proportion of all vertices.  On the other hand, if the expected degree
is smaller than one, namely sub-critical phase, then there is no
``giant'' component and all components are small.

The study of random graphs has a rich history beginning with the
pioneering work of Erd\"{o}s-R\'enyi in 1960's (see
\cite{Bollobas2001,Janson2011}). In recent years, the theory of random
graphs has been significantly expanded by addition of newer models of
random graphs such as the preferential attachment model, configuration
model, inhomogeneous random graphs et al. (see
\cite{Vander2016,Vander2018} for a thorough review of the subject). In
\cite{Janson2008aoap,Vander2018} the weak and strong law has been
established for the giant component of a configuration
model. Recently, a more general strong law result for additive
Lipschitz statistics of configuration model has been shown in
\cite{Salez2016} using the interpolation method. For many of these
models, Strong law or Weak law of large numbers for wide-range of
statistics can often be proven now using local-weak convergence
(\cite{Aldous2004}) or the interpolation method
(\cite{Salez2016}). Additionally, strong law for susceptibility of the
configuration model has been shown via branching process approximation
in \cite{Janson2010}.

However, central limit theorems for statistics of sparse random graphs
especially non Erd\"{o}s-R\'enyi models are harder to find.  Central
limit theorem's for the size of the giant component for
Erd\"{o}s-R\'enyi graphs in the super-critical phase have been studied
in \cite{Pittel1990tree,Barraez2000,Nachmias2007component,Bollobas2012asymptotic}.
Similar limit theorems for $k$-core in the super-critical phase and
susceptibility in the sub-critical phase have been studied in
\cite{JansonLuczak2008aoap} and \cite{Janson2008jmp}
respectively. Asymptotic normality for subtree counts in the sparse
regime (and subgraph counts in the other regimes) has been shown in
\cite{Rucinski1988small} and extensions of the same to functional
limit theorems has been shown in \cite{Janson1990functional}. Also,
alternate proofs of the same via cumulant method and discrete
Malliavin-Stein method can be found in \cite{Feray2016} and
\cite{Krokowski17} respectively.

As for the configuration model: a large deviation result for the
empirical neighbourhood distribution was shown in \cite{Bordenave2015}
using the framework of local weak convergence; a central limit theorem
for self-loops and multiple edges in configuration model was shown in
\cite{Angel2016limit} using the Chen-Stein method for the Poisson
approximation; in \cite{BallNeal17}, asymptotic variance of the giant
component of the configuration model was determined under appropriate
conditions on the degree sequence; and more recently in
\cite{Barbour2017}, this was extended to a central limit theorem as a
consequence of a more general normal approximation result derived for
local statistics of the configuration model using the Stein's method.

Recently, after the first version of this article appeared on the arxiv, in \cite{Janson2018} expectation and variance asymptotics as well as a central limit theorem for various statistics have been established using moment methods. Further, in \cite{Janson2019}, sufficient conditions were given to extend central limit theorems for statistics of configuration model to those of the configuration model conditioned on being a simple graph.

As mentioned earlier, strong law of large numbers for additive
Lipschitz functions with some additional assumptions has been shown in
\cite{Salez2016}. In this article we consider additive statistics in
the sub-critical configuration model and formulate broad assumptions
that need to be verified for a central limit theorem to hold. The
assumptions, apart from subcriticality include a decay rate for the
size of the largest component along with moment bounds and variance
lower bounds for the statistic. We first provide a construction via
edge exploration of { the configuration model} $G_n
:=G(n,\{d^n_i\}_{i=1}^n)$ for a specified degree sequence and prove a
central limit theorem in Theorem \ref{thm:main}.  The exploration can
be used to generate a martingale array sequence from the statistic for
which we verify McLeish's martingale-difference array central limit
theorem, \cite{McLeish1974}. {As an explicit application of our general central limit theorem, we show central limit theorem for number of components isomorphic to a finite tree in a sub-critical configuration with bounded degree sequence (see Remark 
\ref{rem:applns}(1)). We use the variance asymptotics in the recent preprint of \cite{Janson2018} to verify the non-trivial variance lower bound. We also use the results of \cite{Janson2019} to extend some of our central limit theorems to the configuration model conditioned on being a simple graph (see \ref{rem:applns}(4))}. We believe that our generic central limit theorem complements those of \cite{Barbour2017} and \cite{Janson2018}.

The rest of the paper is organized as follows : In Section
\ref{sec:modeldef}, we state our model precisely as well as state our central limit theorem. In Section
\ref{sec:Applns}, we discuss some examples and applications of our
result. We conclude with the proof of our main result in Section \ref{sec:proofs}.

\section{Model and Main Result}  
\label{sec:modeldef}

We begin by defining the statistic of interest. Let
$${\mathcal G} = \{ G : G \mbox{ is a finite multi-graph} \} $$
denote the collection of all finite multi-graphs. Our statistic is defined as a function $F : \mathcal{G} \to \bR$ such that $F$ is invariant under {\em graph isomorphism} (i.e., $F(G_1) = F(G_2)$ if $G_1 \cong G_2$) and {\em additive} (i.e., $F(G) = F(G_1) +F(G_2)$ if $G$ is a vertex-disjoint union of $G_1$ and $G_2$).

We shall construct a random (multi-)graph $G_n:= G(n,
\{d^n_i\}_{i=1}^n)$ on the vertex set $V_n = [n]=\{1,\ldots,n\}$  with
a specified degree sequence $d^n_1,\ldots,d^n_n$ as follows. Assume
that $ \sum_{i=1}^n d^n_i$ is even and we shall denote $ 2m_n =  \sum_{i=1}^n d^n_i$. We shall denote $d^n_{max}
:= \max_{1 \leq i \leq n}d^n_i$. We shall follow the standard
construction via half-edges but with a breadth-first exploration and
hence we shall describe the same in detail. Also, for convenience of
reading we will refer to $d_i^n$ as $d_i$ without the superscript $n$.

Let $W(i) = \{(i,1),\ldots,(i,d_i)\}$ be the set of ordered half-edges
incident on $i \in [n]$ and $\H_n := \cup_{i=1}^nW(i)$ be the total
collection of half-edges on all the vertices in $[n]$. The half-edges
are ordered as per the lexicographic order i.e., $(i,k) \leq (j,l)$ if
$i \leq j$ or $i = j$ and $k \leq l$. For all $t \in [m_n]$,
$\H_n$ is partitioned into three sets $A_t,C_t,U_t$ respectively the
set of {\em active, connected} and {\em unexplored half-edges} with
the initial configuration being $A_1 = W(1), C_1 = \emptyset, U_1 =
\H_n \setminus W(1), X_{1,n} = \emptyset.$ Also our exploration shall
ensure that both of the following events cannot happen for any $k \in
[n]$ : $W(k) \cap A_t \neq \emptyset, U_t \cap W(k) \neq \emptyset$.

{\bf Edge Exploration:} The exploration algorithm can be described as
follows : At step $t \in [m_n]$, choose the smallest edge
$(v_t,i_t)$ (w.r.t. lexicographic order) in $A_t$ and pair it
uniformly at random with one of the other half-edges in $A_t \cup
U_t$, say $(j_t,l_t)$. Then if:
$$
\begin{array}{ll}
  \mbox{$(j_t,l_t) \in A_t$,}& \mbox{set $A_{t+1} = A_t
    - \{(v_t,i_t),(j_t,l_t)\}$ and $U_{t+1} = U_t$;}\\
  &\\
  \mbox{$(j_t,l_t) \notin A_t$,}& \mbox{set $A_{t+1} = A_t \cup W(j_t) - \{(v_t,i_t),(j_t,l_t)\}$ and $U_{t+1} = U_t \setminus W(j_t-1)$.}
\end{array}$$

{Set of connected half-edges at time $t+1$ to be $C_{t+1} = C_t \cup \{(v_t,i_t),(j_t,l_t)\}$ and the newly formed edge is denoted by $X_{t,n} := [(v_t,i_t),(j_t,l_t)].$ Now, repeat the algorithm until $A_{t+1} =  U_{t+1} = \emptyset.$ Since we are pairing two half-edges at every time-step, the algorithm will stop at $t = m_n$.

Denoting the set of matchings on $\H_n$ by $\M_n$, we have generated a
sequence of random elements of $2^{\M_n}$ - $X_{1,n},\ldots,X_{m_n,n}.$ 
Given the sequence $X_{1,n},\ldots,X_{m_n,n}$, we construct the
(multi-)graph $G_n$ by placing an edge between $i,j \in [n]$ for every
pairing of half-edges $[(i,h),(j,l)] \in \X_n :=
(\cup_{k=1}^{m_n}X_{k,n})$. Thus $G_n \in \sigma( \{X_{1,n}, X_{2,n},
\ldots, X_{m_n,n}\})$. Further, we define
$\F_{k,n} = \sigma(\{X_{1,n}, X_{2,n}, \ldots, X_{k,n}\}).$\\

Let $F_n \equiv F(G_n)$ and for $1\leq k \leq m_n$, let $ \Delta_{k,n}
= \E(F_n \mid {\mathcal F}_{k,n}) - \E(F_n \mid {\mathcal F}_{k-1,n})$
with $$ C_n := \sup_{1 \leq k \leq m_n} \E(\Delta_{k,n}^4).$$

We shall make the following assumptions: Let $C^n_{\max}$ denote the
largest connected component in $G_n$. For some $\kappa \geq 0$ and a
sequence $\alpha_n \geq 1$
 \[ \mbox{$C_n(d^n_{max})^2\alpha_n = o(n^{2\kappa})$ and  
$C_n(d^n_{max})^2\P( \mid C^{n}_{\max} \mid > \alpha_n) = o(n^{2\kappa-1}).$} \hspace{1in} (\mbox{G1})\] 
\[  \V{F_n} = \Omega(n^{\frac{1}{2} + \kappa}). \hspace{3in} (\mbox{F1})\]
Now, we are ready to state our main result
\begin{theorem} \label{thm:main} Assume (\mbox{G1}) and (\mbox{F1}).  Then
  \begin{equation}
\frac{F_n - \E(F_n)}{\sqrt{\V{F_n}}} \cid  Z,
  \end{equation}
  with $Z$ being a standard Normal random variable.
\end{theorem}

We note that in view of the assumptions in Theorem \ref{thm:main}, we
can get the following bounds on variance using \eqref{e:var} and
definition of $C_n$ as well as (G1) : There exists a constant $M$ such
that for any $\epsilon > 0$ and large enough $n$,
 \begin{equation}
 \label{e:varub1}
 \sqrt{nC_n\alpha_n}d^n_{max} \leq \epsilon n^{1/2+\kappa} \leq \epsilon M \V{F_n} \leq \epsilon M m_n\sqrt{C_n} \leq nd^n_{max}\sqrt{C_n}.
 \end{equation}
 Thus, we can conclude that for our bounds to hold $\alpha_n = o(n)$
 i.e., the configuration model has to be necessarily sub-critical.

 The key tool in the proof of Theorem \ref{thm:main} is McLeish's martingale-difference array central limit theorem, \cite{McLeish1974}. Our inspiration for
 this central limit theorem arose from powerful central limit theorems proven for geometric functionals of Poisson and Bernoulli point processes proved in \cite{Penrose2001aop,Penrose2001aoap}. The advantage with the martingale-difference array central limit theorem is that it reduces the proof of central limit theorem to moment bounds and convergence in probability of the squared martingale-differences. In \cite{Penrose2001aop,Penrose2001aoap}, the latter is achieved by applying ergodic theorem to appropriate ergodic random fields constructed from the functional and the Poisson point process. However, in our case, the model does not have any underlying ergodicity or stationarity and so, to achieve the required $L_1$ convergence of squared martingale-differences, we use the `sub-criticality' of the configuration model and additivity of the functional. We shall comment later in Section \ref{sec:Applns} about verifying the assumptions on the degree sequence and the function $F$. We also remark about possible extensions of our main theorems in Remarks \ref{rem:applns} and \ref{rthm:main}.

 Our proof techniques require us to restrict to sub-critical
 configuration models (i.e., no giant component). {This is in contrast
 to the results in \cite{Barbour2017} and the very recent one of \cite{Janson2018} which apply to the
 super-critical regime as well. However, our mild assumptions on the locality of the statistics as well as that of degree sequence are less restrictive in some applications. For example, the results of \cite{Janson2018} apply only to subtree counts in the sub-critical regime. Also, as mentioned before, the three proof techniques are different. We use martingale-difference array central limit theorem, \cite{Barbour2017} uses Stein's method via Stein couplings and \cite{Janson2018} uses the classical moment method. 
 
We emphasize that the variance lower bound condition shall usually be the most non-trivial condition in this article to verify and we use the recent variance asymptotics of \cite{Janson2018} to show the same in particular examples (see Remark \ref{rem:applns}(1)). However, we would like to mention that similar variance lower bound conditions appear in most general central limit theorems such as those in \cite{Feray2016,Barbour2017}. Usually these variance lower bound conditions are verified in a case-specific manner like in \cite{BallNeal17}. Only in the recent pre-print of \cite{Janson2018}, one can find somewhat general variance asymptotic formulas for sub-tree counts and certain statistics in the super-critical regime that do not depend on the giant component (\cite[Theorems 3.2 and 3.17]{Janson2018}).

To emphasize the non-triviality of the variance lower bound, we would like to mention that the cardinality of a maximum independent set is a Lipschitz functional and satisfies strong law (see \cite{Salez2016}) but it is much tightly concentrated on a random $d$-regular graph (see
 \cite{Ding2016maximum}). Also, the number of multiple edges as well as self-loops are Lipschitz statistics and with variance growing polynomially in $n$ but the variance growth isn't sufficient to verify the assumptions of our theorems. However, a central limit theorem for the same has been shown in \cite[Section 1.3]{Angel2016limit}. Thus, our general central limit theorem can be considered as reducing the task of proving a central limit theorem for many statistics of random graphs to that of proving `reasonable' variance lower bounds.}

\section{Applications}
\label{sec:Applns}

We will prepare for our applications by recalling a couple of lemmas from the literature.
\begin{lemma}(\cite[Theorem 1.1]{Janson2008aoap})
\label{l:G1}
  Let the $D_n$ be the degree of a randomly chosen vertex in $G(n, \{d_i^n\}_{i=1}^n)$. We assume that
  $$ D_n \stackrel{d}{\rightarrow} D, \E[D_n] \rightarrow \E[D] \in (0, \infty), \frac{\E[D_n(D_n-1)]}{\E[D_n]} \rightarrow \frac{\E[D(D-1)]}{\E[D]} \in [0,1) .$$ 
 Also, assume that for some $\gamma > 3$, uniformly in $n,k$
$$ P(D_n \geq  k)  = O(k^{1-\gamma}).$$
    Then, there exists a constant $A$ so that   
$$\P(\mid C^n_{\max} \mid \geq An^{1/(\gamma-1)}) \rightarrow 0 \mbox{ as } n \rightarrow \infty.$$ 
  \end{lemma}
\medskip

Recall that $\M_n$ is the set of matching on $\H_n$. Consider two
matchings $m,m' \in \M_n$. We say $m,m'$ differ by a {\em switching}
and denote it by $m \sim m'$ if there exists $i_1,i_2,i_3,i_4 \in
\H_n$ such that $(i_1,i_2), (i_3,i_4) \in m$ and $m' = m -
\{(i_1,i_2), (i_3,i_4) \} + \{(i_1,i_3), (i_2,i_4)\}$. Since
$(i_1,i_2)$ corresponds to a pairing of half-edges, a switching is
really a switching of a two pairs of half-edges. Given a matching $m
\in \M_n$, we denote the (multi)-graph obtained by pairing of
half-edges matched in $m$ as $G(m)$and we abbreviate $F(G(m))$ by
$F(m)$.
\begin{lemma}
\label{l:G12}
If $F$ is $M$-Lipschitz under switchings (i.e., $|F(m) - F(m')| \leq
M$ for $m \sim m'$) for some $M < \infty$ or if $F$ is $M/4$-Lipschitz
under edge-addition (i.e., $|F(G)-F(G+(i,j))| \leq M/4$ for any graph
$G$ on $[n]$ and $1 \leq i \neq j \leq n$) for some $M < \infty$, then
we have that $\sup_{n \geq 1} C_n \leq M^4.$
\end{lemma}
\begin{proof}
Let  $m \sim m'$ be as above and let $G = G(m), G' = G(m')$.  Further, set $G_1 = G -  \{(i_1,i_2), (i_3,i_4) \}$. Then $G' = G_1 + \{(i_1,i_3), (i_2,i_4)\}.$ Now note that if $F$ is $M/4$-Lipschitz under edge-addition then we have that,
\[ |F(m) - F(m')| = |F(G) - F(G')| \leq |F(G) - F(G_1)| + |F(G_1) - F(G')| \leq M, \]
i.e., $F$ is $M$-Lipschitz under switchings. Thus it is enough to prove the first part of the Lemma i.e., under the assumption of $M$-Lipschitz under switchings. This is  proven in \cite[Proof of Theorem 2.19]{Wormald1999} or see  \cite[Section 7.1.2]{Bordenave2015}. 
\end{proof}

{\bf Examples of statistics that are Lipschitz under edge-addition:}

{\em We provide a few examples of statistics that are Lipschitz under
  edge-addition. Most of them can be found in \cite[Section 1]{Salez2016}. We present them here for completeness sake. 
\begin{enumerate}
\item {\em Number of connected components :} $F(G) := \beta_0(G)$ is the number of connected components. It is easy to see that this satisfies our assumptions of additivity and is $1$-Lipschitz under edge-addition.
\item {\em Components of fixed size :} Let $G$ be a graph with components $\Gamma_1,\ldots,\Gamma_L$ and $|\Gamma_{\cdot}|$ denotes the size (measured in number of vertices or edges) of the components. For $K,p \in \mathbb{N}$, define
$$F_{p,k}(G) = \sum_{i=1}^L |\Gamma_i|^p \1[ |\Gamma_i| \leq K].$$
For example, $F_{0,K}$ is the number of components of size at most $K$, $F_1(K)$ is the total size of components of size at most $K$. We note that $F_{p,K}$ is $(2K)^p$-Lipschitz. We shall consider an untruncated version (i.e., $K = \infty$) and hence non-Lipschitz later in Corollary \ref{cor:suscep}.
\item {\em Maximum cut-size :}    A cut $(S,T)$ is a partition of the vertex set into two disjoint sets(i.e. vertex set is the disjoint union of  $S$ and $T$). The size $c(S,T)$ of a cut $(S,T)$ is the number of edges between $S$ and $T$. Define $F(G) = \max\{ c(S,T): (S,T) \mbox{ is a cut of } G\}.$ It can be verified that $F$ is 1-Lipschitz under edge-addition.
  
\item {\em Log-partition function :} Let $S$ be a finite set with one map $h : S \to (0,\infty)$  and a second symmetric map $J : S \times S \to (0,\infty).$ Define the statistic
$$ F(G) := \log (\sum_{\sigma \in S^V} w(\sigma))$$ where
$$w(\sigma) := \prod_{v \in V}h(\sigma_v)\prod_{(v,w) \in E}J(\sigma_v,\sigma_w)$$
where we have used the notation that $G = (V,E)$  and $\sigma := \{\sigma_v\}_{v \in V} \in S^V$. It can be easily verified that the above $F$ is additive and $M$-Lipschitz under edge-addition where $M = \max_{s,t \in S}J(s,t).$ 

Statistics such as above arise often in statistical physics where $\sigma$ is said to be the configuration of a system on the graph $G$, $J$ is interpreted to encode pairwise interaction between vertices, $h$ is considered as the external field and the statistics $F$ is called as the {\em log-partition function}. 

Two particular case of special interest that arise from statistical mechanics are : (i) {\em Ising model :}  $S = \{+1,-1\}$ and $J(s,t) = e^{-\beta st}$ for $\beta \geq 0$ and (ii) {\em Potts model :}  $S = \{1,\ldots,q\}$ and $J(s,t) = 1_{s \neq t} + e^{-\beta}1_{s=t}$ for $\beta \geq 0$. Again, in statistical physics terminology, $\beta$ is known as {\em the inverse temperature}.
\end{enumerate}
}

For Lipschitz functions, we now present some corollaries to our main theorem. The first is an easy consequence of Lemma \ref{l:G12} and our main theorem \ref{thm:main}. 
\begin{corollary}
\label{prop:example1} 
Let $\{d^n_i\}_{i=1}^n$ be a degree sequence and $G_n := G(n,\{d^n_i\}_{i=1}^n)$ be the corresponding configuration model. Let $F$ be invariant under graph isomorphisms, additive and  $M$-Lipschitz under switchings. Suppose that for some $\kappa \geq 0$, 
  \begin{enumerate}
    \item[(i)] $d^n_{max} = O(n^{\beta})$ for some $0 \leq \beta \leq \kappa$ 
    \item[(ii)]
\begin{equation} \label{conf:comp}
\P( \mid C^{n}_{\max} \mid > An^{1/{(\gamma-1)}}) = o(n^{2\kappa - 2\beta - 1})
\end{equation}
for some constant $A$ and  $\gamma > (2\kappa -2\beta)^{-1} + 1,$  and
\item[(iii)] $\V{F_n} = \Omega(n^{\frac{1}{2} + \kappa})$.
  \end{enumerate}
  then
  $$\frac{F_n - \E(F_n)}{\sqrt{\V{(F_n)}}} \cid Z.$$ 
\end{corollary}
 \begin{corollary}
\label{cor:intro}
Let $\{d^n_i\}_{i=1}^n$ be a bounded degree sequence and let $G_n := G(n,\{d^n_i\}_{i=1}^n)$ be the corresponding configuration model. Let $F$ be invariant under graph isomorphisms, additive and  $M$-Lipschitz under switchings. Let the $D_n$ be the degree of a randomly chosen vertex in $G(n, \{d_i^n\}_{i=1}^n)$. Suppose that 
\begin{enumerate}
\item[(i)] as $n \rightarrow \infty$,  $$ D_n \stackrel{d}{\rightarrow} D, \E[D_n] \rightarrow \E[D] \in (0, \infty), \frac{\E[D_n(D_n-1)]}{\E[D_n]} \rightarrow \frac{\E[D(D-1)]}{\E[D]} \in [0,1) .$$
\item[(ii)] $\V{F(G_n)} = \Omega(n)$
\end{enumerate}  
then
$$\frac{F_n - \E(F_n)}{\sqrt{\V{F_n}}} \cid  Z,$$
with $Z$ being a standard Normal random variable.
\end{corollary}

 \begin{proof}
 If $\{d_i^n\}_{i=1}^n$ satisfies the assumptions of Lemma \ref{l:G1} and $\sup_{n \geq 1} d^n_{max} < \infty$, then Lemma \ref{l:G1} implies that (\ref{conf:comp}) holds with $\kappa = \frac{1}{2}$ and for any $\gamma >3$. Hence (G1) holds with $\kappa = \frac{1}{2}$ and $\alpha_n = A n^{\frac{1}{\gamma-1}}$ for any $\gamma >3$. In other words, central limit theorem for Lipschitz functionals of bounded degree graphs follows if we show that $\V{F_n} = \Omega(n)$. This completes the proof.
   \end{proof}

 \begin{remark}
\label{rem:applns}

\begin{enumerate}
\item Let $T$ be a finite tree with atleast two vertices and $F_T(G)$ be the number of components of $G$ isomorphic to $T$. Note that from \cite[Theorem 3.2]{Janson2018}  assumption (ii) in Corollary \ref{cor:intro} holds for $F_T(G_n)$ whenever assumption (i) holds and $\mbox{deg}_T(v) \in \{k: P(D=k) >0\}$ for all $v \in T$ with $\mbox{deg}_T(v)$ being the degree of $v$ in $T$.

\item Suppose that $F$ is $M$-Lipschitz and the degree sequence satisfies the conditions in Lemma \ref{l:G1}. Then, from \eqref{e:varub1}, we obtain the bounds that 

 $$  \frac{3}{2 (\gamma -1)}  \leq \kappa \leq \frac{1}{2} + \frac{1}{\gamma-1}.$$ 

Note that such bounds are possible when $\gamma > 2$.	

\item Let us assume again that $F$ is a Lipschitz function under switchings.  Suppose the assumptions of Lemma \ref{l:G1} hold with some $\gamma > 3$. Then we have that $d^n_{max} = O(n^{\frac{1}{\gamma-1}})$ (see \cite[Section
  3.4]{Vander2018}). Thus, we get that  Lemma
\ref{l:G12} implies that (G1) holds with $\alpha_n = An^{\frac{1}{\gamma-1}}$ and for some $\kappa \leq 1/2 + 1/{(\gamma-1)}$ provided we have that

\[  \P( \mid C^{n}_{\max} \mid > An^{1/{(\gamma-1)}}) = o(n^{2\kappa - 2/(\gamma-1) - 1}) \] 

The upper bound on $\kappa$ is justified due to the above remark. Thus, CLT for Lipschitz functionals of such graphs hold if $\V{F_n} = \Omega(n^{1/2+\kappa})$ with $\kappa$ as above.

\item Suppose the assumptions of Corollary \ref{prop:example1} or
  Corollary \ref{cor:intro} hold. Let $G_n^\prime \stackrel{d}{=} G_n
  \mid G_n \mbox{ is a simple graph}$ i.e., $G_n^\prime$ is the configuration model conditioned to be simple. It is a well known fact that
  $G_n^\prime$ is a random simple graph with the uniform distribution
  over all graphs with the given degree sequence (see \cite[Proposition 7.15]{Vander2016}). If $F_n^\prime =
  F(G_n^\prime)$ then we have that
  $$\frac{F_n^\prime - \E(F_n)}{\sqrt{\V{F_n}}} \cid  Z,$$
  with $Z$ being a standard Normal random variable.  The result
  follows from \cite[Corollary 2.3 and Theorem 3.2]{Janson2019}. The
  proof is an imitation of the proof given in Example 8.3 of
  \cite{Janson2019}. The key change in the proof is to replace the upper bound in \cite[(8.7)]{Janson2019} with $MS$ instead of $2S$.
\end{enumerate}

\end{remark}

 An important example of a non-Lipschitz additive statistics for a
 random graph is {\em Susceptibility}. Namely, for a graph $G$, with $n$ vertices and  $K$ connected components, define for $p \geq 0$
 $$ S_p(G) = \sum_{i=1}^K (\mbox{Size of $i$-th component})^p.$$ Note
 that $S_0(G)$ is the number of connected components, $S_1(G) = n$,
 $S_2(G)$ is called Susceptibility and $S_p(G)$ is not in general
 Lipschitz for $p \geq 2$. As mentioned before strong law for
 $S_2(G(n,\{d^n_i\}_{i=1}^n))$ (with $G(n,\{d^n_i\}_{i=1}^n)$ being
 the configuration model) has been shown in \cite{Janson2010} and a
 central limit theorem for subcritical Erd\"{o}s-R\'enyi graphs has
 been shown in \cite{Janson2008jmp}. We now present a corollary that
 provides assumptions under which a central limit theorem holds for
 $S_p(G(n,\{d^n_i\}_{i=1}^n))$ when $p \geq 2$.

 \begin{corollary}
 \label{cor:suscep}
   Let $p \geq 2$. Let $\{d_i^n\}_{i=1}^n$ satisfy assumptions of Lemma \ref{l:G1} and  $G_n := G(n,\{d^n_i\}_{i=1}^n)$ be the corresponding configuration model. Assume that 
 \begin{enumerate}
 \item[(i)] $\gamma > 4p+4$.
 \item[(ii)]$ \P( \mid C^{n}_{\max} \mid > An^{1/{(\gamma-1)}}) = o(n^{-\frac{a}{\gamma- 1}})$ for some $A >0$ and $a > (4p-1)(\gamma-1) + 3$.
\item[(iii)]  $\V{S_p(G_n)} = \Omega(n)$ 
 \end{enumerate} 
Then, 
$$\frac{S_p(G_n) - \E(S_p(G_n))}{\sqrt{\V{S_p(G_n)}}} \cid  Z,$$
with $Z$ being a standard Normal random variable.
 \end{corollary}
 \begin{proof} It is easy to see using (ii) and definition of $C_n$ for Theorem \ref{thm:main} that 
   \begin{equation*}C_n \leq c_P E([C^n_{\mbox{max}}]^{4p})   \leq c_p [ n^{\frac{4p}{(\gamma-1)}} +  n^{4p-\frac{a}{(\gamma- 1)}}  ]  \leq c_p [ n^{\frac{4p}{(\gamma-1)}} +  n^{\frac{4p(\gamma -1) -a}{(\gamma- 1)}}] \end{equation*}
   Note that $d^n_{\mbox{max}} = O(n^{\frac{1}{\gamma -1}})$ and with $\alpha_n = A n^{\frac{1}{\gamma -1}}.$ We have
   $$  C_n(d^n_{\mbox{max}})^2 \alpha_n \leq  c_2 [ n^{\frac{4p+3}{\gamma-1}} + n^{\frac{4p(\gamma-1)+ 3 -a}{(\gamma- 1)}}] =o(n)$$
   and
   $$C_n(d^n_{max})^2\P( \mid C^{n}_{\max} \mid > \alpha_n) \leq  c_3 [ n^{\frac{4p+2-a}{(\gamma-1)}} +  n^{\frac{4p(\gamma - 1) + 2 -2a}{(\gamma- 1)}}] = o(1).$$
 Hence the conditions of Theorem \ref{thm:main} hold with $\kappa = 1/2$ and the normal convergence follows. 

   \end{proof}

\section{Proof of the main result}
\label{sec:proofs}

{\em Proof of Theorem \ref{thm:main} :} Let $X_{1,n},\ldots,X_{m_n,n}$
be the sequence of matchings of half-edges generated by the
edge-exploration process as defined in Section
\ref{sec:modeldef}. For the exploration process at step $t$, define the unexplored vertex set as $\mathcal{U}_t = \{ k \in [n] : W_k \cap U_t \neq \emptyset.\}$.  Denote the explored vertex-set as $\mathcal{E}_t := [n] \setminus \mathcal{U}_t$. Observe that the sequence $X_{1,n},\ldots,X_{m_n,n}$ is not identically distributed but have the following independence property that will be used crucially by us. 

\medskip
{\bf Independence property : } {\it Conditioned on $A_t = 0$, we have that $G(n,\{d^n_i\}_{i=1}^n)$ is a disjoint union of 
$G^1_n := G(|\mathcal{U}_t|,\{d^n_i\}_{i \in \mathcal{U}_t})$ and $G^2_n := G(|\mathcal{E}_t|,\{d^n_i\}_{i \in \mathcal{E}_t})$ and $G^1_n$ and $G^2_n$ are independent.}

We refer the reader to \cite[Lemma 3.3]{Barbour2017} for  a proof of the above. In other words, when there are no active half-edges, the configuration
model becomes a union of two independent configurations models - one
on the connected vertex set and the other on unexplored vertex set.

Clearly, $F_n \in \sigma \{ X_{1,n}, \ldots,
X_{m_n,n}\}$.  For $1 \leq k \leq m_n, n \geq 1$ let ${\mathcal F}_{k,n}
= \sigma \{ X_{1,n}, X_{2,n} \ldots X_{k-1,n}, X_{k,n}\}$ and set
$\mathcal{F}_{0.n} = \emptyset$. Observe that
$$ F_n - \E(F_n)  = \sum_{k=1}^{m_n} \E(F_n \mid {\mathcal F}_{k,n}) - \E(F_n \mid {\mathcal F}_{k-1,n}).$$
and $\Delta_{k,n} = \E(F_n \mid {\mathcal F}_{k,n}) - \E(F_n \mid {\mathcal F}_{k-1,n})$ is a martingale difference sequence. Thus to prove the central limit theorem, we shall verify the conditions of the central limit theorem for martingale difference arrays due to McLeish. Namely, if $D_{k,n} = \frac{\Delta_{k,n}}{\sqrt{\V{F_n}}}$ and
\begin{eqnarray}
  && \sup_{n \geq 1}\EXP{\max_{k \leq m_n}|D_{k,n}|^2)} < \infty  \label{mc1}\\
  && \max_{k \leq m_n}D_{k,n} \cip 0,\label{mc2}\\
  &&    \sum_{k=1}^nD^2_{k,n} \cip 1  \label{mc3} 
\end{eqnarray}
then it follows from  \cite[Theorem 2.3] {McLeish1974} that \begin{equation} \label{mc4}
  \sum_{k=1}^{m_n}D_{k,n} \cid  Z\end{equation}
  with $Z$ being a standard Normal random variable. As  $ \frac{F_n - \E(F_n)}{\sqrt{\V{F_n}}}= \sum_{k=1}^nD_{k,n}$ we would  have the result. To complete the proof we will verify (\ref{mc1}), (\ref{mc2}), and (\ref{mc3}).

{\bf Verifying (\ref{mc1}):} By orthogonality of martingale differences, we have that
\begin{equation}
\label{e:var}
 \V{F_n} = \sum_{k=1}^{m_n} \E(\Delta_{k,n}^2)
\end{equation}
and this implies $$ \sup_{n \geq 1}\EXP{\max_{k \leq m_n}|D_{k,n}|^2)}
\leq \sup_{n \geq 1}\sum_{k \leq m_n} \EXP{|D_{k,n}|^2)}= \sup_{n \geq
  1}\sum_{k \leq m_n} \frac{\EXP{\Delta^2_{k,n}}}{\V{F_n}} = 1.$$

{\bf Verifying (\ref{mc2}):} By the trivial bound that $m_n \leq nd^n_{max}$, we have that for any $\epsilon > 0$,
\[ \P(\max_{k \leq m_n}|\Delta_{k,n}| \geq \epsilon \sqrt{\V{F_n}}) \leq \sum_{k=1}^{m_n}\frac{\epsilon^{-4}\EXP{|\Delta_{k,n}|^4}}{{\V{F_n}}^{2}} \leq C_nm_nn^{-1-2\kappa}\epsilon^{-4} \leq C_nd^n_{max}n^{-2\kappa}\epsilon^{-4}.\]
By (G1) and the fact that $\alpha_n \geq 1$ the above implies that $$\max_{k \leq m_n}|D_{k,n}| \cip 0. $$

{\bf Verifying (\ref{mc3}):} We are left to verify is the convergence
in probability of squared martingale differences.
For $1 \leq k \leq m_n$, define $$E_{n,k}  = \{A_t = \emptyset \mbox{ for some } t \in [k-2d_{\mbox{max}}\alpha_{n}, k] \}$$
$$W_{k,n} = D^2_{k,n}\1_{E_{n,k}} \mbox{ and }  Z_{k,n} = D^2_{k,n} - W_{k,n}.$$
By (\ref{e:var}), $$\sum_{k=1}^{m_n} \EXP{W_{k,n}} +  \sum_{k=1}^{m_n} \EXP{Z_{k,n}} = 1 .$$
Using the above, the triangle inequality and non-negativity of $Z_{k,n}$'s, we have
\begin{eqnarray}\label{term0}
  \mid \sum_{k=1}^{m_n} D^2_{k,n} - 1 \mid & \leq &  \mid \sum_{k=1}^{m_n} W_{k,n} - 1 \mid  +  \mid \sum_{k=1}^{m_n} Z_{k,n} \mid \no\\
  &\leq & \mid \sum_{k=1}^{m_n} W_{k,n} - \sum_{k=1}^{m_n} \EXP{W_{k,n}}  \mid  +  \sum_{k=1}^{m_n} \EXP{Z_{k,n}}  + \sum_{k=1}^{m_n} Z_{k,n} \no\\
  &=& I + II +III.
\end{eqnarray}
We shall now show that each of the terms I and III goes to zero in mean. The latter fact will imply that II goes to zero. We begin with I. By Cauchy-Schwarz inequality,

\begin{eqnarray*}
[\EXP{I}]^2 & \leq &\EXP{I^2} = \V{ \sum_{k=1}^{m_n} W_{k,n}}\\
&=&  \sum_{k=1}^{m_n} \V{W_{k,n}} +  2 \sum_{k=1}^{m_n} \sum_{h=k+1}^{k+2d^n_{max}\alpha_n} \C{W_{k,n}, W_{h,n}} +   2 \sum_{k=1}^{m_n} \sum_{h=k+2d^n_{max}\alpha_n+1}^{m_n} \C{W_{k,n}, W_{h,n}}
\end{eqnarray*}
Now, by the independence property stated at the beginning of the proof and additivity of $F$, we have that $W_{k,n} \in \sigma \{X_{t,n}: t \in [k-2d^n_{\mbox{max}}\alpha_n,k] \}$. Thus, $W_{k,n}$ is independent of $W_{h,n}$ for all $h > k + 2d_{\mbox{max}}\alpha_n.$ This along with the Cauchy-Schwarz inequality will imply 
\begin{eqnarray}
  [\EXP{I}]^2 & \leq & \sum_{k=1}^{m_n} \V{W_{k,n}} +  2 \sum_{k=1}^{m_n} \sum_{h=k+1}^{k+2d^n_{max}\alpha_n} \C{W_{k,n}, W_{h,n}} + 0 \no \\
   & \leq & \sum_{k=1}^{m_n} \EXP{W^2_{k,n}}+  2 \sum_{k=1}^{m_n} \sum_{h=k+1}^{k+2d^n_{max}\alpha_n} \sqrt{\EXP{W_{k,n}^2}\EXP{W_{h,n}^2}} \no \\
   & \leq & \sum_{k=1}^{m_n} \frac{\EXP{\Delta^4_{k,n}}}{\V{F_n}^2} +  2 \sum_{k=1}^{m_n} \sum_{h=k+1}^{k+2d^n_{max}\alpha_n} \frac{\sqrt{\EXP{\Delta_{k,n}^4}\EXP{\Delta_{h,n}^4 }}}{\V{F_n}^2} \no \\
  & \leq & \frac{m_nC_n + 4C_nm_nd^n_{max}\alpha_n }{\V{F_n}^2}\no \\
  &\leq& \frac{nd^n_{\mbox{max}}C_n + 4n (d^n_{max})^2 C_n\alpha_n }{\V{F_n}^2} \no
\end{eqnarray}
Using the first assumption of (G1) and (F1), we have that
 \begin{equation} \label{termI}
\EXP{I} \rightarrow 0 \mbox{ as } n \rightarrow \infty.
\end{equation}
For term II, again by {Cauchy-Schwarz} inequality
\begin{eqnarray*}
  \EXP{II}&=& \EXP{\sum_{k=1}^{m_n} Z_{k,n}}= \sum_{k=1}^{m_n} \EXP{ \frac{\Delta^2_{k,n} \1_{(E_{n,k})^c}}{\V{F_n}}}\\
  &\leq & \sqrt{\frac{m^2_n C_n \P((E_{n,k})^c)}{\V{F_n}^2}}\\
    &\leq & \sqrt{\frac{n^2 (d^n_{\mbox{max}})^2 C_n \P(C^{\mbox{max}}_n \geq \alpha_n)}{\V{F_n}^2}}
  \end{eqnarray*}
Therefore, by the latter assumption of (G1) and (F1) we have
\begin{equation} \label{termII}
\EXP{II} \rightarrow 0 \mbox{ as } n \rightarrow \infty.
\end{equation}
As noted earlier this implies that 
\begin{equation} \label{termIII}
III \rightarrow 0 \mbox{ as } n \rightarrow \infty.
\end{equation}

By (\ref{term0}, \ref{termI}, \ref{termII}, \ref{termIII}) we have that
as $n \rightarrow \infty$, $$  \sum_{k=1}^n D^2_{k,n} \stackrel{L_1}{\rightarrow} 1.$$ 
This completes the proof.
\qed

We conclude with the following remark on the possible class of models for which Theorem \ref{thm:main} holds.

\begin{remark} \label{rthm:main}
As seen from the proof of Theorem \ref{thm:main} the key tool was the
martingale central limit theorem. For this  we used two key properties from the construction of the model:
\begin{enumerate}
  \item[(a)] The filtration ${\mathcal F}_{k,n} = \sigma \{ X_{1,n},  X_{2,n} \ldots X_{k-1,n}, X_{k,n}\}$ generated by the sequence of
    matchings $X_{1,n},\ldots,X_{m_n,n}$ has appropriate `independence' property as stated in the beginning of the proof.
    
  \item[(b)] The statistic is additive and the graph is subcritical to ensure `fast enough decoupling' of the martingale-difference array sequence induced by the above filtration and the statistic.  
\end{enumerate}
 Thus if a random graph model can be constructed using matchings that satisfy (a) and (b) above then for any additive statistic satisfying (F1) and (G1) we can prove a central limit theorem
 \end{remark}

\section*{Acknowledgments}
The authors are thankful to  Andrew Barbour and Adrian R{\"o}llin for sharing an earlier draft of their pre-print.

\noindent  {\bf Siva Athreya}\\
8th Mile Mysore Road, Indian Statistical Institute,
         Bangalore 560059, India.\\
         Email: \texttt{athreya@isibang.ac.in}

         \noindent  {\bf D. Yogeshwaran}\\
8th Mile Mysore Road, Indian Statistical Institute,
         Bangalore 560059, India.\\
         Email: \texttt{d.yogesh@isibang.ac.in}

\end{document}